\documentclass{amsart}
\usepackage{longtable,verbatim}

\def\Hom{\mathrm{Hom}}
\def\Kbar{\overline{K}}

\def\Gal{\mathrm{Gal}}

\def\Pic{\mathrm{Pic}}
\theoremstyle{plain}

\def\varrhobar{\varrho}

\usepackage{amssymb,amsmath,amsfonts,amsthm,epsfig,amscd,stmaryrd}
\usepackage{stmaryrd}
\usepackage[all,cmtip,poly]{xy}
\usepackage{color}

\usepackage{hyperref}

\newtheorem{theorem}[]{Theorem}

\newtheorem{lemma}[]{Lemma}

\theoremstyle{definition}

\newcommand{\C}{\mathbf{C}}
\newcommand{\F}{\mathbf{F}}

\newcommand{\Q}{\mathbf{Q}}
\newcommand{\Z}{\mathbf{Z}}
\newcommand{\Fbar}{\overline{\F}}
\newcommand{\Qbar}{\overline{\Q}}
\newcommand\im{\mathrm{im}}
\newcommand\rhobar{\overline{\rho}}

\DeclareMathOperator{\GL}{GL}
\DeclareMathOperator{\PGL}{PGL}
\DeclareMathOperator{\GSp}{GSp}
\DeclareMathOperator{\Sp}{Sp}
\DeclareMathOperator{\PSp}{PSp}

\def\AA{\mathcal{A}}

\title[Rationality of twists of~$\AA_2(3)$]{Rationality of twists of the Siegel modular variety of genus $2$ and level $3$}

\author[F. Calegari]{Frank Calegari}  \email{fcale@uchicago.edu} \address{The University of Chicago,
5734 S University Ave,
Chicago, IL 60637, USA}

\author[S. Chidambaram]{Shiva Chidambaram}  \email{shivac@mit.edu} \address{Massachusetts Institute of Technology,
77 Massachusetts Ave. 
Cambridge, MA 02139, USA}

\thanks{Both authors were supported in part by NSF Grants
  DMS-1701703. The first author was supported in part by DMS-2001097, and the second author was
  supported in part by the Simons Foundation (grant 550033).}
  
\begin{document}

\begin{abstract} Let~$\rhobar: G_{\Q} \rightarrow \GSp_4(\F_3)$ be a continuous Galois representation
with cyclotomic similitude character. Equivalently, consider $\rhobar$ to be the Galois representation associated to the~$3$-torsion
of a principally polarized abelian surface~$A/\Q$. We prove that the moduli space~$\AA_2(\rhobar)$
of principally polarized abelian surfaces~$B/\Q$ admitting a symplectic isomorphism~$B[3] \simeq \rhobar$
of Galois representations is never rational over~$\Q$ when~$\rhobar$ is surjective, even though
it is both rational over~$\C$ and unirational over~$\Q$ via a map of degree~$6$.
\end{abstract}

\maketitle

\section{Introduction}

Let~$p$ be a prime and suppose that~$A/\Q$ is an abelian variety of dimension~$g$ with a polarization of degree prime to~$p$.
Associated to the action of the absolute Galois group~$G_{\Q}$ on~$A[p]$ there exists
a Galois representation
$$\rhobar: G_{\Q} \rightarrow \GSp_{2g}(\F_p)$$
such that the corresponding  similitude character is the mod-$p$ cyclotomic 
character~$\varepsilon$. One can ask, conversely, whether any such representation
comes from an abelian variety in infinitely many ways.
When~$g = 1$, this question is well-studied, and has a positive answer exactly for~$p = 2$,~$3$, and~$5$. Indeed, the corresponding twists~$X(\rhobar)$
of the modular curve~$X(p)$ are rational over~$\Q$ for~$p = 2$,~$3$, and~$5$,
and have higher genus for larger~$p$.

In~\cite{BCGP}, this question arose for abelian surfaces ($g = 2$) when~$p = 3$.
(The case~$p = 2$, which is also discussed in that paper,
is understood by analyzing the branch points of the 
hyperelliptic involution.)
Let~$\AA_2(3)$ denote the Siegel modular variety of genus~$2$ and level~$3$. It is the moduli space of principally polarized
abelian surfaces together with a symplectic isomorphism~$A[3] \simeq (\Z/3\Z)^2 \oplus (\mu_3)^2$.
Given a~$\rhobar$ as above, one can form the corresponding moduli space~$\AA_2(\rhobar)$
where now one insists that there is a symplectic isomorphism~$A[3] \simeq V$, where~$V$
is the 
representation space of~$\rhobar$ with its symplectic structure.
The variety~$\AA_2(3)$ is well-known to be birational to the Burkhardt  quartic,
which is rational over~$\Q$ (\cite{Bruin}). It is clear that~$\AA_2(\rhobar)$ is isomorphic to~$\AA_2(3)$ over~$\C$ (and even over the fixed field
of the kernel of~$\rhobar$), and hence~$\AA_2(\rhobar)$ is \emph{geometrically}
rational. If~$\AA_2(\rhobar)$ was in fact \emph{rational} (by which we always
mean rational over the base field), then indeed the answer to the question
above would be positive, just as for elliptic curves when~$p \le 5$.
In~\cite[Prop~10.2.3]{BCGP}, 
a weaker result was established:
The variety~$\AA_2(\rhobar)$ is unirational over~$\Q$ via a map
of degree at most~$6$.
As a consequence, any such~$\rhobar$ \emph{does} arise from (infinitely many)  abelian surfaces.
We refer the reader to~\cite{CCR} which produces explicit polynomials describing the universal family over a rational cover of~$\AA_2(\rhobar)$ of degree $6$.
However, the question as to whether~$\AA_2(\rhobar)$ was actually rational was left open. We address
this question here.

\begin{theorem} Let~$\rhobar: G_{\Q} \rightarrow \GSp_4(\F_3)$ be a representation with cyclotomic similitude character.
	Suppose that the order of~$\im(\rhobar)$ is greater than~$96$. Then~$\AA_2(\rhobar)$ is not rational over~$\Q$.
\end{theorem}

More refined results can be extracted directly from the table in~\S\ref{sec:table}. Since~$\rhobar$ has cyclotomic similitude character, the restriction of~$\rhobar$ to~$G_E$, where~$E = \Q(\sqrt{-3})$, has image contained in~$\Sp_4(\F_3)$. If we let~$H$ denote the projection of~$\im(\rhobar|_{G_E})$ to the simple group~$\PSp_4(\F_3)$, then we prove that~$\AA_2(\rhobar)$ is not rational over~$\Q$ for all 
but~$26$
of the~$116$ conjugacy classes
of subgroups of~$\PSp_4(\F_3)$. With the exception of three
cases (including when~$H$ is trivial) where the methods of~\cite{Bruin}
may be applied (see~\S\ref{extra}),  we do not know what happens
in the 
remaining~$23$ cases, nor do we even know whether the rationality of~$\AA_2(\rhobar)$ depends only on~$\im(\rhobar)$ or not. One easy remark is that, for a quadratic character~$\chi$, there is an
isomorphism~$\AA_2(\rhobar) \simeq \AA_2(\rhobar \otimes \chi)$, and so the rationality of~$\AA_2(\rhobar)$ depends only
on the  image of~$\rhobar |_{G_E}$ in~$\PSp_4(\F_3)$.

The case of a surjective representation~$\rhobar$ is of special interest, since this is what happens generically for the three-torsion Galois representations of abelian surfaces.
\begin{theorem}  \label{theorem:intro} Suppose that~$\rhobar$ is surjective. Then~$\AA_2(\rhobar)$ is not rational over $\Q$,
and the minimal degree of any rational cover is~$6$.
\end{theorem}
In light of the result~\cite[Prop~10.2.3]{BCGP} mentioned above, the constant~$6$ is best possible in this case.

The key ingredient in our results is the explicit description of the cohomology of the compactified Siegel modular variety~$\AA_2^*(3)$ given in~\cite{HW61}. We use it to study the Galois module~$\Pic_{\Qbar}(\AA_2^*(\rhobar))$. The Galois action over~$E = \Q(\sqrt{-3})$ factors through the projectivization of~$\rhobar$ turning it into a $H$-module. We then calculate group cohomology of this module for various subgroups~$P \subset H$, and employ a necessary criterion for rationality (see Theorem~\ref{theorem:manin}) to deduce our results.

\subsection{Acknowledgments}

We thank Jason Starr and Yuri Tschinkel for discussions about rationality versus geometric rationality
for smooth varieties over number fields,  Steven Weintraub for a suggestion on how
to explicitly extract a description of~$H^2(\AA^*_2(3),\Z)$ as a~$G = \PSp_4(\F_3)$-module from
Theorem~4.9 of~\cite{HW61}, and Mark Watkins with help using \texttt{Magma}~\cite{magma}. We thank
the anonymous referees for useful comments and corrections, and we also thank Nils Bruin for explaining to us
many of the ideas in  section~\ref{extra}.

\section{Strategy}

The main idea behind the proof is to follow a strategy employed by Manin for cubic surfaces. Recall~\cite[\S A.1]{manin} that a continuous~$G_K$-module with the discrete topology
is called a \emph{permutation module} if it admits a finite free~$\Z$-basis on which~$G_K$ acts (via
a finite quotient) via permutations, and that two~$G_K$-modules~$M$ and~$N$ are \emph{similar}
if~$M \oplus  P \simeq N \oplus Q$ for some permutation modules~$P$ and~$Q$.
In particular, we employ the following theorem.

\begin{theorem}{\cite[\S A.1 Theorem~2]{manin}}  \label{theorem:manin} Let~$Z$ be a smooth projective algebraic variety over a number
field~$K$.
Suppose that~$Z$ is rational over~$K$. Then~$\Pic_{\Kbar} Z$ as a~$G_K$-module is stably permutation. In other words, it is similar to the zero 
module. 
\end{theorem}
 
The Shimura variety~$\AA_2(3)$ admits a smooth toroidal projective compactification~$\AA^*_2(3)$, the (canonical) toroidal compactification constructed
by Igusa \cite{Igusa1967}. The automorphism group of~$\AA^*_2(3)$ over~$\Qbar$ is the group~$G = \PSp_4(\F_3)$,
the simple group of order~$25920$, which acts over the field~$E = \Q(\sqrt{-3})$.
It will be convenient from this point onwards to always work over the field~$E$. (Certainly rationality over~$\Q$ implies
rationality over~$E$, so non-rationality over~$E$ implies non-rationality over~$\Q$.)
This action on~$\AA_2(3)$ arises explicitly from the action of~$G$ on the~$3$-torsion~$A[3] = (\Z/3 \Z)^2
\oplus (\mu_3)^2 \simeq (\Z/3\Z)^4$ over~$E$.
We will apply Theorem~\ref{theorem:manin} to the corresponding
twist~$\AA^*_2(\rhobar)$.
We then make crucial use of very explicit description of
the cohomology of this compactification given by Hoffman and Weintraub \cite{HW61}. We
recall some facts from that paper here now.

\subsection{Picard group}
\label{Pic}
The Picard group of~$\AA^*_2(3)$ over~$\Qbar$ is a free~$\Z$-module of
rank~$61$. It is generated by two natural sets of classes. The first is a~$40$-dimensional
space explained by the~$40$ connected components of the boundary. The second is a~$45$-dimensional
space explained by divisors coming from Humbert surfaces. These are also
in one to one correspondence with the~$45$ nodes on the Burkhardt  quartic.
Together, these
generate the Picard group of~$\AA^*_2(3)$ over~$\Qbar$, which is free of rank~$61$.
Indeed, the Betti cohomology of~$\AA^*_2(3)$ over~$\Z$ is free of degrees~$1,0,61,0,61,0,1$ for~$i = 0,\ldots 6$ by~\cite[Theorem~1.1]{HW61}. 	
Furthermore, all of these classes are trivial under the action of~$G_{E}$.

Let $\rhobar: G_{\Q} \rightarrow \GSp_4(\F_3)$ be a continuous Galois representation with cyclotomic similitude character. The assumption on the similitude character implies that the restriction of $\rhobar$ to $E$
is valued in~$\Sp_4(\F_3)$. Let
$$\varrhobar: G_{E} \rightarrow G = \PSp_4(\F_3)$$
denote the projectivization of the representation~$\rhobar$ restricted to~$E$.
The group~$G$ acts over~$E$ on~$\AA^*_2(3)$ via automorphisms,
and~$\AA^*_2(\rhobar)$ is the twist of~$\AA^*_2(3)$ by~$\varrhobar$. 
The group~$\Pic_{\Qbar} \AA^*_2(\rhobar)$  as a~$G_E$-module is obtained
by considering~$\Pic_{\Qbar} \AA^*_2(3)$ as a~$G$-module and then obtaining
the Galois action via the map~$\varrhobar: G_{E} \rightarrow G$.
Thus it remains to closely examine~$\Pic_{\Qbar}(\AA^*_2(3))$ as a~$G$-module
over~$\Z$.
In fact, we can quickly prove a weaker version
of Theorem~\ref{theorem:intro} by studying this~$G$-module
over~$\Q$.
The group~$G$ admits a unique conjugacy class~$G_{45}$ of subgroups of index~$45$,
but two conjugacy classes of index~$40$; let~$G_{40}$ denote the (conjugacy class of)
subgroups which fix a point in the tautological action of~$G \subset \PGL_4(\F_3)$ on~$\mathbf{P}^3(\F_3)$.
The following is an easy consequence of the calculations of~\cite{HW61} (and is also confirmed by our
\texttt{Magma} code).

\begin{lemma} As~$\Q[G]$-modules, there is an 
	equality of virtual representations
$$H^2(\AA^*_2(3),\Q)  \simeq \Pic_{\Qbar}(\AA^*_2(3)) \otimes \Q =  \Q[G/G_{40}] + \Q[G/G_{45}]
- [\chi_{24}],$$
where~$\chi_{24} \otimes_{\Q} \C$ is the unique absolutely irreducible~$24$-dimensional representation of~$G$.
\end{lemma}

Now, assuming that~$\varrhobar$ is \emph{surjective}, we can prove that~$\AA^*_2(\rhobar)$
is not rational simply by proving that~$\chi_{24}$ is not virtually equal to
a sum of permutation representations. If $R_{\Q}(G)$ denotes the representation ring of $G$, this is equivalent to proving that~$\chi_{24} \in  R_{\Q}(G)$
does not lie in the 
Burnside subring 
generated by permutation representations. But one may compute 
(using \texttt{Magma} or otherwise) that the Burnside cokernel of~$G$ has order~$2$ and is generated
by~$\chi_{24}$. 
This proves a weaker version of Theorem~\ref{theorem:intro} showing that any rational cover of~$\AA_2(\rhobar)$ should have degree at least $2$, although it is softer in that it only needs the~$\Q[G]$-representation
rather than the~$\Z[G]$-module.
This argument also applies if one only assumes that the image of~$\varrhobar$ is~$H \subset G$, as long as the restriction of~$\chi_{24}$ to~$H$  is still non-trivial in the Burnside cokernel,
which it is for precisely~$8$
of the~$116$ conjugacy classes of subgroups of~$G$.

\subsection{Cohomological Obstructions}
\label{cohomological_obstruction}
From now on, we let $H$ denote the image of~$\varrhobar: G_{E} \rightarrow G = \PSp_4(\F_3)$.
A second way to prove that a 
Galois module is not similar to the zero 
 module
is to use cohomology. If~$M$ is a permutation 
 module of~$H$,  then 
the restriction of~$M$ to any subgroup~$P$ is also a permutation 
 module,
and thus a direct sum of~$P$-modules
of the form~$\Z[P/Q]$ for subgroups~$Q$ of~$P$.
(Note that since a permutation module of a group~$G$ arises from a finite~$G$-set, it always decomposes over~$\Z$ into a direct sum of such irreducible permutation modules.)
Then, Shapiro's Lemma implies that~$H^1(P,M)$ is a direct sum of groups of the form
$$H^1(P,\Z[P/Q]) = H^1(Q,\Z) = 0,$$
where the second group vanishes because~$Q$ is finite. Moreover, the~$\Z$-dual~$M^{\vee} = \Hom(M,\Z)$
of a permutation 
 module is isomorphic to the same permutation 
 module
(a permutation matrix is its own inverse transpose).
Thus one immediately has the following elementary criterion.

\begin{lemma}[Cohomological Criterion for non-rationality] \label{lemma:crit} Let~$M$ denote the~$G$-module $\Pic_{\Qbar}(\AA^*_2(3))$. Suppose~$\AA^*_2(\rhobar)$ is rational over~$E = \Q(\sqrt{-3})$, 
and~$\varrhobar |_{G_E}$ has image~$H \subset G$. Then
$$H^1(P,M^{\vee}) = H^1(P,M) = 0$$
for every subgroup~$P \subset H$.
\end{lemma}

We note that this is not an ``if and only if'' criterion. In the language of~\cite{sansuc},
the lemma is saying that~$M$ as a~$G_E$-module is \emph{flasque} and \emph{coflasque} respectively. In general, this is weaker than being stably permutation
 (which itself is not enough to formally imply rationality).

In order to test this criterion in practice, we need an explicit description of~$M$ as a~$\Z[G]$-module
rather than a~$\Q[G]$-module. In order to do this, we explain how an explicit description of~$M$
can be extracted from 
Theorem~4.9 of~\cite{HW61}. That theorem describes a set of elements which generate
both~$H_4(\AA^*_2(3),\Z)$ and~$H^2(\AA^*_2(3),\Z)$, and explicitly gives the intersection pairing between them.
Moreover, the basis comes with a transparent action of the group~$G$. Specifically,
$H^2(\AA^*_2(3),\Z)$ is given as a quotient of~$\Z[G/G_{40}] \oplus \Z[G/G_{45}]$.
Hence to compute~$H^2(\AA^*_2(3),\Z)$ as a~$G$-module, it suffices to compute the quotient
of~$\Z[G/G_{40}] \oplus \Z[G/G_{45}]$ by the saturated subspace which pairs trivially
with all elements of~$H_4(\AA^*_2(3),\Z)$.
Having carried out this computation, we obtain a 
free abelian group 
of rank~$61$ with an explicit action of~$G$. We then do the following for every
conjugacy class of subgroups~$H \subset G$.
\begin{enumerate}
\item Determine whether~$\chi_{24}$ is non-trivial in the Burnside cokernel of~$H$.
\item Determine whether~$H^1(P,M) \ne 0$ for any subgroup~$P \subset H$.
\item Determine whether~$H^1(P,M^{\vee}) \ne 0$ for any subgroup~$P \subset H$.
\end{enumerate}
If any of these is non-trivial, this proves that~$\AA^*_2(\rhobar)$ is not rational.
Moreover, the computation of these cohomology groups allows us to deduce our result about the minimal degree of any rational covering.

\begin{lemma}
	\label{pullback-pushforward}
	Let~$M$ denote the~$G$-module $\Pic_{\Qbar}(\AA^*_2(3))$. Suppose~$\varrhobar |_{G_E}$ has image~$H \subset G$. Let~$n$ denote the least common multiple of the exponents of~$H^1(P,M)$ and~$H^1(P,M^{\vee})$ as $P$ varies over all subgroups of~$H$. Suppose~$f : X \rightarrow \AA^*_2(\rhobar)$ is a rational cover of degree~$d$
	defined over~$\Q$. Then $n$ divides $d$.
\end{lemma}
\begin{proof}
The induced pullback map $f^* : \Pic_{\Qbar}(\AA^*_2(\rhobar)) \rightarrow \Pic_{\Qbar}(X)$ and pushforward map $f_* : \Pic_{\Qbar}(X) \rightarrow \Pic_{\Qbar}(\AA^*_2(\rhobar))$ are Galois equivariant since $f$ is
defined over~$\Q$. The composite map~$g = f_* \circ f^*$ on $\Pic_{\Qbar}(\AA^*_2(\rhobar))$ is multiplication by $d$. The discussion in~\S\ref{Pic} shows that the~$G_E$-module $\Pic_{\Qbar}(\AA^*_2(\rhobar))$ can be thought of as the~$H$-module $M$.

	By Theorem~\ref{theorem:manin}, we know that~$\Pic_{\Qbar}(X)$ is stably permutation as a Galois module and hence the Galois cohomology group~$H^1(G_{\Q},\Pic_{\Qbar}(X)) = 0$. Therefore, the maps induced by~$g$ on the cohomology groups~$H^1(P,M)$ and~$H^1(P,M^{\vee})$ are the zero maps for every subgroup~$P \subset H$. Since the map~$g$ is multiplication by~$d$, the induced map on cohomology is also multiplication by~$d$, and hence we deduce that the exponent of each of these cohomology groups divides~$d$.
\end{proof}

We give one final statement which can be extracted from \texttt{Magma} using the code given in~\cite{code}, but not directly from the table. In order to represent elements of~$G = \PSp_4(\F_3)$ by
matrices, we follow the conventions of \texttt{Magma} by fixing~$\Sp_4(\F_3) \subset \GL_4(\F_3)$ to be the matrices preserving the symplectic form
$$J=\left( \begin{matrix} 0 & 0 & 0 & 1 \\ 0 & 0 & 1 & 0 \\ 0  & -1 & 0 & 0 \\-1 & 0 & 0 & 0 \end{matrix} \right).$$

\begin{lemma} 
Suppose that the image of~$\rhobar$ contains an element conjugate in~$\PSp_4(\F_3)$ to
$$\left( \begin{matrix} 1 & 0 & 0 & -1 \\ 0 & 1 & 1 & 0 \\ 0 & 0 & 1 & 0 \\ 0 & 0 & 0 & 1 \end{matrix} \right).$$
Then~$\AA_2(\rhobar)$ is not rational,
and the minimal degree of any rational cover is divisible by~$3$.
\end{lemma}

\begin{proof}
It suffices to note that this element generates the subgroup labelled as subgroup~$6$ in the table below,
and then to apply 
Lemma~\ref{pullback-pushforward}.
\end{proof}

\subsection{Other cases where rationality can be established} \label{extra}
The analysis of Baker's parametrization~\cite{Baker} undertaken in~\cite[\S4]{Bruin} allows
one to deduce the rationality of certain twists of
the Burkhardt quartic~$B$ (and hence of~$\AA_2(\rhobar)$) in a few more cases. (We thank
Nils Bruin for pointing this out to us, as well as explaining the geometric construction below.)
The rational parametrization~$\mathbf{P}^3 \dashrightarrow B$  over~$\Q$ constructed in~\cite{Bruin}
is not equivariant
with respect the action of~$\PSp_4(\F_3)$. If it were, then the twists~$\AA_2(\rhobar)$ we are considering
would all be birational to Brauer--Severi varieties. However,
because they are also unirational over~$\Q$ by~\cite[Prop~10.2.3]{BCGP},  they would be rational over~$\Q$,
which we prove in this paper to be false in general. 
On the other hand, the parametrization~$\mathbf{P}^3 \dashrightarrow B$
\emph{is} equivariant with respect to the (unique up to conjugacy) cyclic group of order~$9$~\cite[\S4.3]{Bruin},
and also with respect to the corresponding group scheme over~$\Q$ whose~$E$ points are this
group of order~$9$ (c.f. \cite[\S2.3]{CCR}), which controls the descent from~$E$ to~$\Q$. 
In particular, the same argument implies that~$\AA_2(\rhobar)$ is rational in two further cases, namely, the subgroups labelled~$n = 4$ (of order~$3$)
and~$n = 24$ (of order~$9$) in the table below.  One can also arrive at this rational parametrization
more geometrically, following~\cite[\S4]{Bruin}, whose notation we now freely follow.
The variety of lines~$L_{J_1,J_2,J_3}$ incident with~$3$-distinct planes~$J_i \subset \mathbf{P}^4$
is geometrically rational. If these planes are mutually skew and lie on~$B$, there is a 
dominant map~$L_{J_1,J_2,J_3} \dashrightarrow B$
defined by noting that a line  will  generically intersect~$B$ in four points and each~$J_i$ in one point, and hence
one can send the line to the fourth point of intersection with~$B$. There are~$40$ Jacobi planes~$J_i$ on~$B$, and~$2880$ triples of mutually skew such planes. 
The stabilizer under~$\PSp_4(\F_3)$ on these~$2880$ triples is the
cyclic group of order~$9$.  The assumption that~$H$ is  contained inside this group
then implies that
 there exists a triple of~$\Gal(\Qbar/\Q)$-invariant 
 mutually skew planes on
the twist of~$B$ corresponding to~$\rhobar$.
The result then follows after noting that~$L_{J_1,J_2,J_3}$ is rational over~$\Q$ whenever
this triple is defined over~$\Q$. (We omit a direct proof of this last claim in light of the alternate
argument given above.)

\section{Computation} \label{sec:table}
Let~$M$ denote the~$G$-module $\Pic_{\Qbar}(\AA^*_2(3)) \simeq H^2(\AA^*_2(3),\Z)$.
 We have, by Poincar\'{e} duality, an isomorphism~$M^{\vee} = H^4(\AA^*_2(3),\Z)$.
Below we present in a table the result of our computation for all~$116$ conjugacy classes of subgroups~$H \subset G$, indicating the following data:
\begin{enumerate}
\item  An ordering~$n = 1 \ldots 116$ of the conjugacy class of the subgroup~$H$ as determined by \texttt{Magma}.
\item The group~$H$ in the small groups database~\cite{Small}. The first element of the pair
gives the order of~$H$.
\item The
order  of $M$ in the Burnside cokernel of $H$ over~$\Q$
(if it is non-trivial). If this is greater than~$1$, then the corresponding twist is not rational over~$E$ (or~$\Q$).
\item The least common multiple of the exponents of~$H^1(P,M)$ and~$H^1(P,M^{\vee})$ as~$P$ ranges over subgroups~$P \subset H$.
If this is greater than~$1$, then the corresponding twist is not rational over~$E$ (or~$\Q$).
In particular, the fact that this number is~$6$ for~$G$ itself proves Theorem~\ref{theorem:intro}.
\item The pre-image of~$H$ in $\Sp_4(\F_3)$ acts on~$\F^4_3$. Is this action absolutely irreducible?
(That is, is the action on~$\Fbar^4_3$ irreducible.)
\item A list of the  conjugacy class of maximal subgroups of~$H$ (as indexed in the table). This
allows one to compute the LCM column directly. 
The table is separated into blocks to reflect the geometry of the corresponding poset of subgroups. In particular, all maximal subgroups of~$H$ occur in blocks before that of $H$. 
\item The last two columns give~$H^1(H,M)$ and~$H^1(H,M^{\vee})$.
\end{enumerate}
One must be careful while reading the table because the ordering of the conjugacy classes of subgroups is not canonical. The Small Group tag and the indices of the maximal subgroups given in the second and sixth columns of the table \emph{do}, however, determine the ordering uniquely  once  we distinguish between the conjugacy classes indexed by~$n = 2,3$,~$n = 4,5,6$,~$n = 9,11$, and~$n = 10,12$. This can be done by considering the length of each of these conjugacy classes (i.e., the number of subgroups in each conjugacy class) as shown in the following table.
\begin{center}
\begin{tabular}{|c|c|}
	\hline
	$n$ & Length\\
	\hline
	2 & 45\\
	3 & 270\\
	\hline
	4 & 40\\
	5 & 120\\
	6 & 240\\
	\hline
\end{tabular}
\quad
\begin{tabular}{|c|c|}
	\hline
	$n$ & Length\\
	\hline
	9 & 270\\
	11 & 405\\
	\hline
	10 & 270\\
	12 & 540\\
	\hline
\end{tabular}
\end{center}

The \texttt{Magma} code available at~\cite{code} computes~$G$ and~$M$ directly
from the description given by Hoffman and Weintraub~\cite{HW61}.
This leads to a representation of~$G$ as generated by two 
sparse ~$61 \times 61$ matrices~$x$ and~$y$ in~$\mathrm{GL}_{61}(\Z)$ such that the
underlying module on which~$G$ acts (on the right, by \texttt{Magma} conventions) is~$M$. 
The matrices~$x$ and~$y$ are also printed in the output file of our \texttt{Magma} script. 

 \begin{longtable}{|c|c|c|c|c|c|c|c|}
\hline
$n$ &  SmallGroup & B & LCM &  irred
 & maximal subgroups & $H^1(M)$  &  $H^1(M^{\vee})$ \\
\hline
$1$ & \texttt{<1,1>}  &  & $1$ & no &  & & \\
\hline
$2$ &  \texttt{<2,1>} &       & $1$  & no & 1  & & \\
$3$ &  \texttt{<2,1>} &       & $1$   & no & 1   & & \\
$4$ & \texttt{<3,1>} &       & $1$   & no & 1   & & \\
$5$ & \texttt{<3,1>} &       & $1$   & no  & 1  & & \\
$6$ & \texttt{<3,1>} &       & $3$   & no& 1   & $\Z/3\Z$ & $\Z/3\Z$ \\
$7$ &  \texttt{<5,1>} &       & $1$   & no & 1  & & \\
\hline
$8$ &  \texttt{<4,1>} &       & $1$   & no & 2  & & \\
$9$ &  \texttt{<4,2>} &       & $1$   & no &  2 3  & & \\
$10$ & \texttt{<4,2>} &       & $2$   & no  & 3  & &  $(\Z/2\Z)^2$ \\
$11$ &  \texttt{<4,2>} &       & $2$   & no & 2 3  & &   $\Z/2\Z$ \\
$12$ &  \texttt{<4,2>} &       & $1$   & no & 3  & & \\
$13$ & \texttt{<4,1>} &       & $1$   & no & 3  & & \\
$14$ & \texttt{<6,1>} &       & $3$   & no & 2  6  & $\Z/3\Z$ & \\
$15$ & \texttt{<6,2>} &       & $1$   & no & 2  4   & & \\
$16$ &  \texttt{<6,2>} &       & $3$   & no & 2 6  & & \\
$17$ & \texttt{<6,1>} &       & $3$   & no  & 3  6  & &  $\Z/3\Z$ \\
$18$ &  \texttt{<6,1>} &       & $1$   & no & 3 5   & & \\
$19$ & \texttt{<6,2>} &       & $1$   & no  & 2 5 & & \\
$20$ & \texttt{<6,2>} &       & $1$   & no  &  3 5   & & \\
$21$ & \texttt{<9,2>} &       & $3$   & no &  5 6   & &  $(\Z/3\Z)^2$ \\
$22$ &  \texttt{<9,2>} &       & $3$   & no & 4  6   & &  $(\Z/3\Z)^2$  \\
$23$ &  \texttt{<9,2>} &       & $3$   & no & 4  5  6   & & \\
$24$ &  \texttt{<9,1>} &       & $1$   & no & 4  & & \\
$25$ &  \texttt{<10,1>} &       & $1$   & no & 3 7    & & \\
\hline
$26$ &  \texttt{<8,4>} &       & $1$   & no  & 8  & & \\
$27$ &  \texttt{<8,5>} &       & $2$   & no  & 11 12   & &  $(\Z/2\Z)^2$ \\
$28$ &  \texttt{<8,5>} &       & $2$   & no & 10 11   & & $(\Z/2\Z)^2$ \\
$29$ &  \texttt{<8,5>} &       & $2$   & no  & 9 10 11  & & \\
$30$ &  \texttt{<8,2>} &       & $2$   & no & 8 11   & & \\
$31$ &  \texttt{<8,2>} &       & $2$   & no &  11 13   & $\Z/2\Z$ &  $\Z/2\Z$ \\
$32$ &  \texttt{<8,3>} &       & $2$   & no & 8 11   &  &  $\Z/2\Z$ \\
$33$ &  \texttt{<8,3>} &       & $2$   & no & 10 12 13  & & $\Z/2\Z$   \\
$34$ &  \texttt{<8,3>} &       & $1$   & no &  9 12 13  & & \\
$35$ &  \texttt{<12,3>} &       & $2$   & no & 5 10   & & \\
$36$ &  \texttt{<12,3>} &       & $3$   & no & 6 12   & $\Z/3\Z$ &  $\Z/3\Z$ \\
$37$ &  \texttt{<12,4>} &       & $3$   & no & 9 14 16 17   & & \\
$38$ &  \texttt{<12,5>} &       & $1$   & no &  9 19 20   & & \\
$39$ &  \texttt{<12,1>} &       & $1$   & no & 13 20  & & \\
$40$ &  \texttt{<12,2>} &       & $1$   & no &  8 15  & & \\
$41$ &  \texttt{<12,4>} &       & $1$   & no &  12 18 20   & & \\
$42$ &  \texttt{<18,4>} &       & $3$   & no & 17 18 21  & & $(\Z/3\Z)^2$  \\
$43$ &  \texttt{<18,3>} &       & $3$   & no & 14 16 21  & & \\
$44$ &  \texttt{<18,3>} &       & $3$   & no &14 19 23   & & \\
$45$ &  \texttt{<18,3>} &       & $3$   & no & 14 15 22   & & \\
$46$ &  \texttt{<18,3>} &       & $3$   & no & 18 20 21   & & $\Z/3\Z$ \\
$47$ &  \texttt{<18,3>} &       & $3$   & no & 17 20 23   & & \\
$48$ &  \texttt{<18,5>} &       & $3$   & no & 15 16 19 23   & & \\
$49$ &  \texttt{<20,3>} &       & $1$   & yes &   13 25  & & \\
$50$ &  \texttt{<27,5>} &       & $3$   & no & 21 22 23  & & $\Z/3\Z$ \\
$51$ &  \texttt{<27,3>} &       & $3$   & no & 22   & & $(\Z/3\Z)^2$  \\
$52$ &  \texttt{<27,4>} &       & $3$   & no & 22 24  & &  $\Z/3\Z$  \\
\hline
$53$ &  \texttt{<16,14>} &       & $2$   & yes & 28 29   & & \\
$54$ &  \texttt{<16,13>} &       & $2$   & no & 26 30 32   & & \\
$55$ &  \texttt{<16,11>} &       & $2$   & yes & 27 28 30 32   & & $\Z/2\Z$  \\
$56$ &  \texttt{<16,3>} &       & $2$   & no & 28  31  & &  $(\Z/2\Z)^2$  \\
$57$ &  \texttt{<16,11>} &       & $2$   & yes &  27 29 31 33 34   & & $\Z/2\Z$  \\
$58$ &  \texttt{<16,3>} &       & $2$   & no &  29 30 31   & & \\
$59$ &  \texttt{<24,3>} &       & $1$   & no & 15 26   & & \\
$60$ &  \texttt{<24,13>} &       & $2$   & no & 20 29 35   & & \\
$61$ &  \texttt{<24,3>} &       & $3$   & no & 16 26   & & \\
$62$ &  \texttt{<24,3>} &       & $1$   & no &  19 26  & & \\
$63$ &  \texttt{<24,11>} &   $2$   & $1$   & no & 26 40   & & \\
$64$ &  \texttt{<24,13>} &       & $2$   & no &  19 28 35   & & \\
$65$ &  \texttt{<24,13>} &       & $6$   & no &   16 27 36  & & \\
$66$ &  \texttt{<24,12>} &       & $2$   & no & 18 33 35   & & \\
$67$ &  \texttt{<24,12>} &       & $6$   & no &  17 33 36   & & $\Z/6\Z$   \\
$68$ &  \texttt{<24,12>} &       & $3$   & no &  14 34 36   & $\Z/3\Z$ & \\
$69$ &  \texttt{<24,8>} &       & $1$   & no & 34 38 39 41   & & \\
$70$ &  \texttt{<36,10>} &       & $3$   & no & 37 42 43   & & \\
$71$ &  \texttt{<36,10>} &       & $3$   & no & 41 42 46   & &  $\Z/3\Z$  \\
$72$ &  \texttt{<36,9>} &       & $3$   & no & 13 42   & &  $\Z/3\Z$  \\
$73$ &  \texttt{<36,12>} &       & $3$   & no & 37 38 44 47 48   & & \\
$74$ &  \texttt{<54,8>} &       & $3$   & no &   45 51  & & \\
$75$ &  \texttt{<54,13>} &       & $3$   & no &42 46 47 50   & & $\Z/3\Z$  \\
$76$ &  \texttt{<54,12>} &       & $3$   & no & 43 44 45 48 50  & & \\
$77$ &  \texttt{<60,5>} &       & $2$   & no & 18 25 35   & & \\
$78$ &  \texttt{<60,5>} &       & $3$   & no & 17 25 36   & & $\Z/3\Z$  \\
$79$ &  \texttt{<81,7>} &       & $3$   & no & 50 51 52  & & $\Z/3\Z$  \\
\hline
$80$ &  \texttt{<32,49>} &       & $2$   & no & 54 56   & &  \\
$81$ &  \texttt{<32,6>} &       & $2$   & yes &   55 56   & & $\Z/2\Z$ \\
$82$ &  \texttt{<32,27>} &       & $2$   & yes & 53 55 56 57 58  & & \\
$83$ &  \texttt{<48,30>} &       & $2$   & no & 39 58 60   & & \\
$84$ &  \texttt{<48,49>} &       & $2$   & yes & 38 53 60 64  & & \\
$85$ &  \texttt{<48,33>} &       & $2$   & yes & 40 54 59  & & \\
$86$ &  \texttt{<48,48>} &       & $2$   & no &  41 57 60 66  & &  $\Z/2\Z$ \\
$87$ &  \texttt{<48,48>} &       & $6$   & yes &  37 57 65 67 68  & & \\
$88$ &  \texttt{<72,40>} &     & $3$   & no & 34 70 71 72  & & \\
$89$ &  \texttt{<72,25>} &    $2$   & $3$   & no & 48 59 61 62 63   & & \\
$90$ &  \texttt{<80,49>} &       & $2$   & yes &  7 53  & & \\
$91$ &  \texttt{<108,40>} &       & $3$   & no & 71 75   & & $\Z/3\Z$  \\
$92$ &  \texttt{<108,15>} &       & $3$   & no & 40 74  & & \\
$93$ &  \texttt{<108,38>} &       & $3$   & no & 70 73 75 76  & & \\
$94$ &  \texttt{<108,37>} &       & $3$   & no & 39 72 75  & & \\
$95$ &  \texttt{<120,34>} &       & $3$   & yes & 37 49 68 78  & & \\
$96$ &  \texttt{<120,34>} &       & $2$   & yes &  41 49 66 77  & & \\
$97$ & \texttt{<162,10>} &       & $3$   & no & 74 76 79  & & \\
\hline
$98$ &  \texttt{<64,138>} &       & $2$   & yes &  80 81 82  & & \\
$99$ &  \texttt{<96,204>} &       & $2$   & no & 62 64 80  & & \\
$100$ & \texttt{<96,204>} &       & $6$   & no & 61 65 80   & & \\
$101$ & \texttt{<96,201>} &    $2$   & $2$   & no & 63 80 85   & & \\
$102$ & \texttt{<96,195>} &       & $2$   & yes & 69 82 83 84 86   & & \\
$103$ & \texttt{<160,234>} &       & $2$   & yes & 25 82 90   & & \\
$104$ & \texttt{<216,88>} &  $2$     & $3$   & no & 63 92  & & \\
$105$ & \texttt{<216,158>} &       & $3$   & no &  69 88 91 93 94  & & \\
$106$ & \texttt{<324,160>} &       & $3$   & no &  36 79 91   & & $\Z/3\Z$  \\
$107$ & \texttt{<360,118>} &       & $6$   & no &   66 67 72 77 78  & & $\Z/3\Z$  \\
\hline
$108$ & \texttt{<192,1493>} &       & $6$   & yes & 87 98 100   & & \\
$109$ & \texttt{<192,201>} &       & $2$   & yes & 84 98 99   & & \\
$110$ & \texttt{<288,860>} &    $2$   & $6$   & no &  89 99 100 101  & & \\
$111$ &  \texttt{<648,533>} &   $2$    & $3$   & no & 89 97 104   & & \\
$112$ &  \texttt{<648,704>} &       & $3$   & no & 68 97 105 106   & & \\
$113$ &  \texttt{<720,763>} &       & $6$   & yes & 86 87 88 95 96 107  & & \\
\hline
$114$ &  \texttt{<576,8277>} &   $2$    & $6$   & yes & 73 108 109 110  & & \\
$115$ &  \texttt{<960,11358>} &       & $2$   & yes &  77 102 103 109  & & \\
\hline
$116$ &  $G$ &    $2$     & $6$   & yes &  111 112 113 114 115  & & \\
\hline
\end{longtable}

\bibliographystyle{alpha}
\bibliography{Rationality}

\end{document}